\documentclass[a4paper,11pt]{scrartcl}
\usepackage{amsmath,amssymb,latexsym}
\usepackage{amsthm}

\usepackage{tgpagella}
\usepackage[T1]{fontenc}

\setkomafont{disposition}{\normalcolor\bfseries}
\addtokomafont{section}{\centering\normalfont\scshape}
\addtokomafont{subsection}{\normalfont\scshape}

\usepackage[UKenglish]{babel}
\usepackage[utf8]{inputenc}
\usepackage{bussproofs}
\usepackage{mathrsfs}%
\usepackage{xparse}

\usepackage{hyperref}
\hypersetup{pdfauthor={Graham E. Leigh},pdftitle={Conservativity for theories of compositional truth via cut elimination}}

\providecommand\lang\mathcal
\providecommand\theory\mathsf 
\providecommand\texttheory\textsf
\providecommand\pred\mathrm 
\providecommand\textpred\texttt
\providecommand\func\mathit 
\providecommand\struc\mathfrak
\providecommand\set\mathrm
\renewcommand\b\boldsymbol
\providecommand\L{\mathcal L}
\providecommand\PA{\theory{PA}}
\providecommand\ID{\theory{I\Delta_0}}
\providecommand\IDexp{\ID+\mathsf{exp}}
\providecommand\CT{\theory{CT}}
\providecommand\Sent{\pred{Sent}}
\providecommand\Term{\pred{Term}}
\providecommand\Form{\pred{Form}}
\providecommand\Ax{\pred{Ax}}
\providecommand\Bew{\pred{Bew}}

\providecommand\Var{\mathop\mathrm{Var}}%
\newcommand*\subdot[1]{\oalign{$#1$\cr\hfil.\hfil}}%
\providecommand\gn[1]{\ulcorner #1 \urcorner} %
\providecommand\dvee%
		{\mathbin{\subdot \vee}}
\providecommand\dwedge%
		{\mathbin{\subdot \wedge}}
\providecommand\drightarrow%
		{\mathbin{\subdot \rightarrow}}
\providecommand\dnot%
		{\mathord{\subdot{\lnot}}}
\providecommand\dbot%
		{\mathord{\subdot{\bot}}}
\providecommand\dforall%
		{\mathord{\subdot\forall}}
\providecommand\dexists%
		{\mathord{\subdot\exists}}
\providecommand\sub{\func{sub}}
\providecommand\subn{\func{subn}}

\providecommand\Seq\Rightarrow
\renewcommand\L{\lang L}
\renewcommand\phi\varphi
\providecommand\size[1]{\lvert #1 \rvert}
\providecommand\pair[1]{\langle#1\rangle}
\providecommand\concat{{^\frown}}
\providecommand\T{\pred T}
\renewcommand\P{\pred P}
\providecommand\lh[1]{lh(#1)}%
\makeatletter
\let\oldlabel\label
\def\mylabel#1{\in@{tab:}{#1} \ifin@ (#1)\oldlabel{#1} \else\marginpar{\tiny(#1)}\oldlabel{#1} \fi}
\let\oldmarginpar\marginpar
\def\mysidenote#1{\oldmarginpar{\tiny #1}}
\newcommand\sidenote[1]{}
\def\IsDraft{\let\label\mylabel\let\sidenote\mysidenote}%

\makeatother
\renewcommand\pred\mathsf%
\renewcommand\T{\mathrm T}%
\renewcommand\c[1]{\mathrel{\check#1}}%
\newtheorem{theorem}{Theorem}
\newtheorem{lemma}{Lemma}
\newtheorem{corollary}{Corollary}
\newtheorem{definition}{Definition}
\renewcommand\d{\subdot}%
\newcommand\val{\pred {val}}%
\renewcommand\Var{\pred{Var}}%
\newcommand\s\mathsf%
\providecommand\Q{\pred Q}%
\RenewDocumentCommand\CT { s }{\IfBooleanTF{#1}{\theory{CT}^*}{\theory{CT}}}%
\RenewDocumentCommand\c {m}{#1}
\RenewDocumentCommand\theory {m} {\mathsf{#1}}%
\RenewDocumentCommand \P {} {p}

\let\myRightLabel\RightLabel
\DeclareDocumentCommand\RightLabel {m} {\myRightLabel{\small #1}}
\let\myLeftLabel\LeftLabel
\DeclareDocumentCommand\LeftLabel {m} {\myLeftLabel{\small #1}}
\title{Conservativity for theories of compositional truth via cut elimination}
\author{Graham E.\ Leigh}
\begin{document}

\maketitle
\begin{abstract}
  We present a cut elimination argument that witnesses the conservativity of the compositional axioms for truth (without the extended induction axiom) over any theory interpreting a weak subsystem of arithmetic. In doing so we also fix a critical error in Halbach's original presentation. Our methods show that the admission of these axioms determines a hyper-exponential reduction in the size of derivations of truth-free statements.
\end{abstract}
\def\MainFile{\jobname}
\ifx\MainFile\undefined
   \input{Deflating.Truth.arxiv.tex}
   \stop
\fi
\section{Overview}\label{sec:Overview}
We denote by $\IDexp$ and $\IDexp_1$ the first-order theories extending Robinson's arithmetic by $\Delta_0$-induction and, respectively, axioms expressing the totality of the exponentiation and hyper-exponentiation function.
If $\theory S$ is a recursively axiomatised first-order theory interpreting $\IDexp$ then by $\CT[\theory S]$ we denote the extension of $\theory S$ by a fresh unary predicate $\T$ and the {\em compositional axioms of truth} for $\T$.\footnote{See definition~\ref{defn:CT} for a formal definition of $\CT[\theory S]$.}

In this paper we provide syntactic proofs for the following theorems. 
\begin{theorem}\label{thm:1}\label{thm:CTConserve}
  Let $\theory S$ be an elementary axiomatised theory interpreting $\IDexp$.
  Every theorem of $\CT[\theory S]$ that does not contain the predicate $\T$ is a theorem of $\theory S$. Moreover, this fact is verifiable in $\IDexp_1$.
\end{theorem}
Let $p$ be a fresh unary predicate symbol not present in the language $\L$ of $\theory S$.
An $\L$-formula $\pred D$ is an {\em $\theory S$-schema} if $\theory S\vdash \pred D\gn\sigma \rightarrow \sigma$ for every $\L$-formula $\sigma$ and there exists a finite set $U$ of $\L\cup\{p\}$-formulæ with at most $x$ free such that $\theory S\vdash \pred D x\rightarrow \exists \psi\,\bigvee_{\phi\in U}(x=\gn{\phi[\psi/p]})$.
\begin{theorem}\label{thm:2}\label{thm:CTIndConserve}
    Let $\theory S$ be an elementary $\L$-theory interpreting $\IDexp$. For any $\theory S$-schema $\pred D$, the theory $\CT[\theory S]+\forall x(\pred Dx\rightarrow \T x)$
    is a conservative extension of $\theory S$.
\end{theorem}
In the case that $\theory S$ is Peano arithmetic, the first part of both theorems is a consequence of the main theorems of \cite{KKL81,Lac81}. 
The proof is model-theoretic, however, establishing that a countable non-standard model of Peano arithmetic contains a full satisfaction class if and only if it is recursively saturated. 
Since every model of the Peano axioms is elementarily extended by a recursively saturated model, proof-theoretic conservativity is obtained.
Halbach~\cite{Hal99} offers a proof-theoretic approach to the first part of theorem~\ref{thm:1}. 
The strategy proceeds as follows. First the theory $\CT[\theory S]$ is reformulated as a finitary sequent calculus with a cut rule and rules of inference corresponding to each of the compositional axioms for truth.
A typical derivation in this calculus will involve cuts on formulæ involving the truth predicate.
The elimination of all cuts is not possible as $\theory S$ is assumed to interpret a modicum of arithmetic. Instead, Halbach outlines a method of partial cut elimination whereby every cut on a formula involving the truth predicate is systematically replaced by a derivation without cuts on formulæ containing $\T$.
As noted in \cite{EV13b} and \cite{Hal12} however, the proof contains a critical error. 
An inspection of the cut elimination argument demonstrates that it does provide a method to eliminate cuts on formulæ of the form $\T s$ provided there is a separate derivation, within say $\theory S$, establishing that the logical complexity of the formula coded by $s$ is bounded by some numeral.

The present paper provides the necessary link between the $\CT[\theory S]$ and its fragment with bounded cuts. This takes the form of the following lemma.
\begin{lemma}[Bounding lemma]\label{lem:BoundingIntro}
    If $\Gamma$ and $\Delta$ are finite sets consisting of only truth-free and atomic formulæ, and the sequent $\Gamma\Seq\Delta$ is derivable in $\CT[\theory S]$, then there exists a derivation of this sequent in which all cuts are either on $\L$-formulæ or bounded.
\end{lemma}
Let $\CT^*[\theory S]$ denote the subsystem of $\CT[\theory S]$ featuring only bounded cuts.
Since this calculus permits the elimination of all cuts containing the truth predicate, the first part of theorem~\ref{thm:1} is a consequence of the above lemma. 
Moreover, the proof (see §\ref{sec:main}) yields bounds on the size of the resulting derivation, from which the second part of theorem~\ref{thm:1} can be deduced. 

A particular instance of theorem~\ref{thm:2} of interest is if $\pred D$ is the predicate $\Ax_\theory S$ formalising the property of encoding an axiom of $\theory S$.
In this case we notice that the reduction of $\CT[\theory S]$ to $\CT^*[\theory S]$ also yields a reduction of $\CT[\theory S]+\forall x(\Ax_\theory Sx\rightarrow \T x)$ to a corresponding extension of $\CT*[\theory S]$. 
Unlike before, the latter theory does not admit cut elimination. 
Instead we show that the extension of $\CT*[\theory S]$ is relatively interpretable in $\CT[\theory S]$,
whence theorem~\ref{thm:1} provides the desired result.

Theorem~\ref{thm:1} has been independently proved by Enayat and Visser in \cite{EV13b} (the special case in which $\theory S$ is Peano arithmetic is also outlined in \cite{EV13a}). Their proof involves a refinement and extension of the original model-theoretic proof appearing in \cite{KKL81} that permits the argument to be formalised within a weak fragment of arithmetic.
The author understands that Enayat and Visser also have a proof of theorem~\ref{thm:2}, again model-theoretic, though at the time of writing this is not in circulation.

\subsection{Outline}\label{sec:outline}
In the following two sections we formally define the theory $\CT[\theory S]$ for a theory $\theory S$ interpreting $\IDexp$ and its presentation as a sequent calculus, as well as the sub-theory with bounded cuts, $\CT^*[\theory S]$.
Section~\ref{sec:Approximations} contains the technical lemmata necessary to prove the core theorems, that every theorem of $\CT[\theory S]$ not involving the predicate $\T$ is derivable in $\CT^*[\theory S]$; the proofs of which form the content of section~\ref{sec:main}.
In the final section we present applications of our analysis to questions relating to interpretability and speed-up.

\section{Preliminaries}\label{sec:prelim}
We are interested in first-order theories that possess the mathematical resources to develop their own meta-theory. It is well-known that only a weak fragment of arithmetic is required for this task, namely $\ID+\exp$. 
For our purposes we therefore take the interpretability of $\ID+\exp$ as representing that a theory possess the resources to express basic properties about its own syntax. 
For notational convenience we shall restrict ourselves exclusively to theories that extend this base theory. 
Our results, however, apply just as well to the general case.

Let $\L$ be a recursive, first-order language containing the language of arithmetic. 
It will be useful to work with an extension of $\L$ that includes a countable list of fresh predicate symbols $\{\P^i_j\mid i,j<\omega\}$ where $\P^i_j$ has arity $i$, plus a fresh propositional constant $\epsilon$. 
We denote this extended language by $\L^+$.
We fix some standard representation of $\L^+$ in $\IDexp$, which takes the form of a fixed simple Gödel coding of $\L^+$ into $\L$ with:
\begin{enumerate}
  \item 
    Predicates $\Term_\L x$, $\Form_\L x$, $\Sent_\L x$, and $\Var x$ of $\L$ expressing respectively the relations that $x$ is the code of a closed term, a formula, a sentence and a variable symbol of $\L^+$.
  \item 
    A $\Sigma_1$-predicate $\val(x,y)$ such that $\val(\gn t,t)$ is provable in the base theory for every term $t$. 
    We view $\val$ as defining a function
    and write $\pred{eq}(r,s)$ in place of $\forall x\forall y(\val(r,x)\wedge\val(s,y)\rightarrow x\c=y)$.
  \item 
    Predicates defining operations on codes; namely the binary terms $\d=$, $\dwedge$, $\dvee$, $\drightarrow$, $\dforall$, $\dexists$, $\d\P$, unary terms $\d \Q$ for each relation $\Q$ in $\L$ and $\d d$, and a ternary term $\sub$ with:
  \begin{itemize}
    \item $\d \Q(\pair{\gn{t_1},\ldots,\gn{t_n}})=\gn{\Q(t_1,\ldots,t_n)}$ for each $\Q\in\L$,
    \item $\d \P(\bar\jmath,\pair{\gn{t_1},\ldots,\gn{t_i}})=\gn{\P^i_j(t_1,\ldots,t_i)}$, 
    \item $\d d(\gn \alpha)=x$ if the logical complexity of the $\L^+$ formula $\alpha$ is $x$, and
    \item 
        $\sub (x,y,z)$ denoting the usual substitution function that replaces in the term or formula (encoded by) $x$ each occurrence of the variable with code $y$ by the term with code $z$.
        We abbreviate uses of this function by writing $x[z/y]$ in place of $\sub(x,y,z)$.
  \end{itemize}
\end{enumerate}

\begin{definition}\label{defn:CT}
  Let $\theory S$ be some fixed theory in a recursive language $\L$ which interprets $\IDexp$. 
  The theory $\CT[\theory S]$ is formulated in the language $\L_\T=\L\cup\{\T\}$ and consists of the axioms of $\theory S$ together with
  \begin{align*}
    \Term_\L x\wedge\Term_\L y&\rightarrow( \T x\d=y\leftrightarrow x^\circ\c=y^\circ)),\\
    \Sent_\L x\wedge\Sent_\L y&\rightarrow(\T(x\dwedge y)\leftrightarrow \T x\wedge \T y),\\
    \Sent_\L x\wedge\Sent_\L y&\rightarrow(\T(x\dvee y)\leftrightarrow \T x\vee\T y),\\
    \Sent_\L x\wedge\Sent_\L y&\rightarrow(\T(x\drightarrow y)\leftrightarrow (\T x\rightarrow\T y)),\\
    \Sent_\L x&\rightarrow(\T\dnot x\leftrightarrow \lnot\T x),\\
    \Var y\wedge \Sent_\L x(\dot {\bar0}/y)&\rightarrow (\T\dforall yx\leftrightarrow\forall z(\Term_\L z\rightarrow\T x(z/y))),\\
    \Var y\wedge \Sent_\L x(\dot {\bar0}/y)&\rightarrow (\T\dexists yx\leftrightarrow\exists z(\Term_\L z\wedge\T x(z/y))),\\
    \Term_\L x_1\wedge\cdots\wedge\Term_\L x_n&\rightarrow( \T(\d \Q\pair{x_1,\ldots,x_n})\leftrightarrow \Q(\val x_1,\ldots,\val x_n)).
  \end{align*}
  for each relation $\Q$ of $\L$ (with arity $n$).
  We call the formulæ above the {\em compositional axioms for $\L$} and any formula in the language of $\L$ {\em arithmetical}.
  Moreover explicit mention of the base theory $\theory S$ is often omitted  and we write $\CT$ and $\CT^*$ in place of $\CT[\theory S]$ and $\CT^*[\theory S]$ respectively.
\end{definition}

Finally, we fix a few notational conventions for the remainder of the paper.
The start of the Greek lower-case alphabet, $\alpha$, $\beta$, $\gamma$, etc., will be used to represent formulæ of $\L_\T=\L\cup\{\T\}$, while the end, $\phi$, $\chi$, $\psi$, $\omega$, as well as Roman lower-case symbols $r$, $s$, etc.~denote terms in $\L$.\footnote{The former list will be used exclusively as meta-variables ranging over terms encoding formulæ of $\L^+$.}
Upper-case Greek letters $\Gamma$, $\Delta$, $\Sigma$ etc., are for finite sets of $\L_\T$ formulæ and boldface lower-case Greek symbols $\b\phi$, $\b\psi$, etc.~represent finite sequences of $\L$ terms. 
For a sequence $\b\phi=(\phi_0,\dots,\phi_k)$, $\T\b\phi$ denotes the set $\{\T\phi_i\mid i\le k\}$. 
As usual, $\Gamma,\alpha$ is shorthand for $\Gamma\cup\{\alpha\}$ and $\Gamma,\Delta$ for $\Gamma\cup\Delta$.

\section{Two sequent calculi for compositional truth}\label{sec:seqcal}
Let $\theory S$ be a fixed theory extending $\ID+\exp$ formulated in the language $\L$.
We present sequent calculi for $\CT[\theory S]$ and $\CT^*[\theory S]$.
In the former calculus, derivations are finite and the calculus supports the elimination of all cuts on non-atomic formulæ containing the truth predicate. 
The latter system replaces the cut rule of $\CT[\theory S]$ by two restricted variants: one of these is the ordinary cut rule applicable to only formulæ not containing $\T$; the other is a cut rule for the atomic truth predicate which is only applicable if the formula under the truth predicate subject to the cut has, provably, a fixed finite logical complexity.
This second variant turns out to be admissible, so any sequent derivable in $\CT^*[\theory S]$ has a derivation containing only arithmetical cuts.
It follows therefore, that $\CT^*[\theory S]$ is a conservative extension of $\theory S$.
We show that any $\CT[\theory S]$ derivation can be transformed into a derivation in $\CT^*[\theory S]$ and hence obtain the conservativity of $\CT[\theory S]$ over $\theory S$.

We now list the axioms and rules of $\CT[\theory S]$ and $\CT^*[\theory S]$.
\subsection{Axioms}
\begin{enumerate}
  \item $\Gamma \Seq \Delta, \phi$ if $\phi$ is an axiom of $\theory S$,
  \item $\Gamma,r\c= s,\T r\Seq \Delta,\T s$ for all terms $r$ and $s$,
  \item $\Gamma,\T r\Seq\Sent (r),\Delta$ for every $r$.
\end{enumerate}
\subsection{Arithmetical rules}
\begin{center}
 \begin{tabular}{cc}
    \AxiomC%
        {$\Gamma\Seq\Delta, \alpha$}%
            \RightLabel{($\forall$R)}
    \UnaryInfC%
        {$\Gamma\Seq\Delta,\forall v_i\alpha$}
    \DisplayProof
    &
    \AxiomC%
        {$\Gamma,\alpha(s/v_i)\Seq\Delta$}%
            \RightLabel{($\forall$L)}
    \UnaryInfC%
        {$\Gamma,\forall v_i\alpha\Seq\Delta$}
    \DisplayProof
    \\\\
    \AxiomC%
        {$\Gamma\Seq\Delta,\alpha, \beta$}%
            \RightLabel{($\vee$R)}
    \UnaryInfC%
        {$\Gamma \Seq\Delta, \alpha\vee \beta$} 
    \DisplayProof
    &
    \AxiomC%
        {$\Gamma, \alpha \Seq\Delta$}%
            \AxiomC%
                {$\Gamma, \beta \Seq\Delta$}
                    \RightLabel{($\vee$L)}
    \BinaryInfC%
        {$\Gamma , \alpha\vee \beta \Seq\Delta$} 
    \DisplayProof
    \\\\
    \AxiomC%
        {$\Gamma, \alpha \Seq\Delta$}%
            \RightLabel{($\lnot$R)}
    \UnaryInfC%
        {$\Gamma \Seq\Delta, \lnot \alpha$}
    \DisplayProof
    &
    \AxiomC%
        {$\Gamma\Seq\Delta, \alpha$}%
            \RightLabel{($\lnot$L)}
    \UnaryInfC%
        {$\Gamma, \lnot \alpha \Seq\Delta$}
    \DisplayProof
  \end{tabular}
\end{center}
\begin{gather*}
    \AxiomC%
        {$\Gamma,\alpha\Seq\Delta$}
        \AxiomC%
        {$\Gamma\Seq\Delta,\alpha$}%
            \LeftLabel{(Cut$_\L$)}
            \RightLabel{provided $\alpha\in\L$}
    \BinaryInfC%
      {$\Gamma\Seq\Delta$}
    \DisplayProof
\end{gather*}
We write $\Gamma \Seq^* \Delta$ to express that the derivation of $\Gamma \Seq \Delta$ involves only the axioms and arithmetical rules.
\subsection{Truth rules}
\begin{center}
 \begin{tabular}{cc}
    \AxiomC%
        {$\Gamma\Seq\Delta,\T\psi_0,\T\psi_1$}%
            \RightLabel{($\vee_\T$R)}
    \UnaryInfC%
        {$\Gamma,\psi\c=\psi_0\dvee\psi_1\Seq\Delta,\T\psi$} 
    \DisplayProof
    &
    \AxiomC%
        {$\Gamma,\T\psi_0\Seq\Delta$}%
            \AxiomC%
                {$\Gamma,\T\psi_1\Seq\Delta$}
                    \RightLabel{($\vee_\T$L)}
    \BinaryInfC%
        {$\Gamma,\psi\c=\psi_0\dvee\psi_1,\T\psi\Seq\Delta$} 
    \DisplayProof
    \\\\
    \AxiomC%
        {$\Gamma\Seq\Delta,\T(\psi_0[v_i/s])$}%
            \RightLabel{($\forall_\T$R)}
    \UnaryInfC%
        {$\Gamma,\psi\c=\dforall s\psi_0\Seq\Delta,\T\psi$}
    \DisplayProof
    &
    \AxiomC%
        {$\Gamma,\T(\psi_0[t/s])\Seq\Delta$}%
            \RightLabel{($\forall_\T$L)}
    \UnaryInfC%
        {$\Gamma,\psi\c=\dforall s\psi_0,\T\psi\Seq\Delta$}
    \DisplayProof
    \\\\
    \AxiomC%
        {$\Gamma,\T\psi_0\Seq\Delta$}%
            \RightLabel{($\lnot_\T$R)}
    \UnaryInfC%
        {$\Gamma,\Sent\psi,\psi\c=\dnot\psi_0\Seq\Delta,\T\psi$}
    \DisplayProof
    &
    \AxiomC%
        {$\Gamma\Seq\Delta,\T\psi_0$}%
            \RightLabel{($\lnot_\T$L)}
    \UnaryInfC%
        {$\Gamma,\Sent\psi,\psi\c=\dnot\psi_0,\T\psi\Seq\Delta$}
    \DisplayProof
    \\\\
    \AxiomC%
        {$\Gamma\Seq\Delta, \pred{eq}(r,s) $}%
            \RightLabel{($=_\T$R)}
    \UnaryInfC%
        {$\Gamma,\phi\c=(r\subdot=s)\Seq\Delta,\T \phi$}
    \DisplayProof   
    & 
    \AxiomC%
        {$\Gamma, \pred{eq}(r,s) \Seq\Delta$}%
            \RightLabel{($=_\T$L)}
    \UnaryInfC%
        {$\Gamma, \phi\c=(r\subdot=s),\T \phi\Seq\Delta$}
    \DisplayProof   
  \end{tabular}
\end{center}
\subsection{Additional cut rules}
In $\CT[\theory S]$:
\begin{align*}
  \AxiomC%
    {$\Gamma,\T \phi \Seq\Delta$}
    \AxiomC%
        {$\Gamma\Seq\Delta,\T \phi$}%
        \LeftLabel{(Cut$_\T$)}
  \BinaryInfC%
      {$\Gamma\Seq\Delta$}
  \DisplayProof
\end{align*}
In $\CT^*[\theory S]$:
\begin{gather*}
    \AxiomC%
        {$\Gamma,\T\phi\Seq\Delta$}
            \AxiomC%
                {$\Gamma\Seq\Delta,\T\phi$}%
                \AxiomC%
                    {$\Gamma,\Sent\phi\Seq^*\subdot d(\phi)\le\bar k$}
                    \LeftLabel{(Cut$^k_\T$)}
    \TrinaryInfC%
        {$\Gamma\Seq\Delta$}
    \DisplayProof
\end{gather*}
Normal eigenvariable conditions apply to four quantifier rules. We refer to the two rules (Cut$_\T$) and (Cut$^m_\T$) collectively as {\em$\mathit T$-cuts}.

\subsection{Derivations}
Derivations in either $\CT[\theory S]$ or $\CT^*[\theory S]$ are defined in the ordinary manner;
the {\em truth depth} of a derivation is the maximum number of truth rules occurring in a path through the derivation.
The {\em truth rank} is the least $r$ such that for any rule (Cut$^m_\T$) occurring in the derivation, $m< r$. The {\em rank} of a derivation is any pair of numbers $(a,r)$ such that $a$ bounds the truth depth and $r$ the truth rank of the derivation.

\subsection{Meta-theorems for \texorpdfstring{$\CT[\theory S]$}{{\sf CT[S]}}}
The key fact we require from $\ID+\exp$ is that the theory suffices to show that codes for $\L^+$ formulæ are uniquely decomposable.
\begin{lemma}[Unique readability lemma]\label{lem:uniquereadability}
The sequent $\Gamma\Seq\Delta$ is derivable in $\ID+\exp$ whenever one of the following conditions hold.
 \begin{enumerate}
  \item $\Gamma$ is a doubleton subset of $\{x\c=y_0\dvee z_0,x\c=\dforall y_1z_1,x\c=(y_2\d=z_2),x\c=\dnot y_3\}$.
  \item $\{\Sent_\L(y),\Sent_\L(z)\}\subset\Gamma$, $\Gamma\cap\{x\c=y\dvee z,x\c=z\dvee y,x\c=\dforall zy,x\c=\dnot y\}\neq \emptyset$ and $\{\Sent_\L(x)\}\subseteq \Delta$.
  \item $\{y_0=y_1\wedge z_0=z_1\} \subseteq \Delta$ and $\Gamma$ extends:
  \begin{enumerate}
   \item $\{x\c=y_0\dvee z_0,x\c=y_1\dvee z_1\}$;
   \item $\{x\c=\dforall y_0z_0,x\c=\dforall y_1z_1\}$; or
   \item $\{x\c=(y_0\d=z_0),x\c=(y_1\d=z_1)\}$.
  \end{enumerate}
  \item $\{y_0 = y_1\} \subseteq \Delta$ and $\{x\c=\dnot y_0,x\c=\dnot y_1\} \subseteq \Gamma$.
  \item $\emptyset \neq \Gamma \subseteq \{x\c=y\dvee z,x\c=z\dvee y,x\c=\dforall zy,x\c=\dnot y\}$ and $\{\d d(y)<\d d(x)\} \subseteq \Delta$.
  \item $\emptyset \neq \Gamma \subseteq \{x\c=y\dvee z,x\c=z\dvee y,x\c=\dforall zy,x\c=\dnot y,x\c=(y\d=z),x\c=(z\d=y)\}$ and $\{y<x\} \subseteq \Delta$.
  \item $\{\d d(x)\le x\}\subseteq \Delta$.
 \end{enumerate}
\end{lemma}

If $\theory S$ does not contain axioms containing the truth predicate then partial cut elimination is at least available in $\CT[\theory S]$.
\begin{lemma}[Embedding lemma for {$\CT$}]\label{thm:1stCutElim}
  Suppose $\T$ does not occur in $\L$ and $\CT[\theory S] \vdash \alpha$. 
  Then the sequent $\emptyset\Seq\alpha$ has a derivation according to the rules of $\CT[\theory S]$.
\end{lemma}
The next lemma demonstrates the key difference between $\CT$ and $\CT^*$.
\begin{lemma}[Cut elimination theorem]\label{thm:2ndCutElim}
  Suppose $\Gamma\Seq\Delta$ is derivable in $\CT^*$ with cut rank $(a,r+1)$. Then the same sequent is derivable with rank $(3^a,r)$.
\end{lemma}
\begin{proof}
    The argument follows the standard cut elimination procedure that is available for the formulation of $\CT$ in $\omega$-logic
    where the standard measure of complexity for terms encoding $\L$-sentences is available. 
    The simplest approach to achieving cut elimination in that setting is through the use of a ``reduction lemma'' formalisation.
    In the finitary scenario, this corresponds to proving that from derivations of the sequents $\Gamma\Seq\Delta,\T\chi$ and $\Gamma,\T\chi\Seq\Delta$, with ranks $(a,r)$ and $(b,r)$ respectively, and a truth-free derivation of the sequent $\Gamma\Seq \subdot d(\chi)\le \bar r$, a derivation of the sequent $\Gamma\Seq\Delta$ can be obtained with rank $((a+b)\cdot2,r)$.
    
    As usual the proof proceeds via induction on the sum of the heights of the two derivations and we can assume that $\T\chi$ is principal in both derivations.
    If either sequent is an axiom, it takes the form $\Gamma',\chi'\c=\chi,\T\chi'\Seq\Delta,\T\chi$, whence substituting $\chi$ for $\chi'$ in the other sequent we obtain $\Gamma\Seq\Delta$. 
    That leaves only the truth rules to consider.
    We will provide only one of the relevant cases of the proof and leave the remainder as an exercise for the reader.
  
    Suppose the first derivation ends with an application of ($\forall_\T$R). Then $a=a'+1$ and there are terms $s_0$ and $\chi_0$ such that the formula $\chi\c=\dforall s_0 \chi_0$ is a member of $\Gamma$ and the sequent
    \[
    \Gamma\Seq \Delta,\T\chi,\T(\chi_0[v_i/s_0])
    \]
    is derivable with rank $(a',r)$.
    Now if any rule other than ($\forall_\T$L) occurs as the last rule in the derivation of $\Gamma,\T\chi\Seq\Delta$, there are terms $\chi_0'$ and $\chi_1'$ such that either $\{\chi\c=\dforall s_0\chi_0,\chi\c=\chi_0'\dvee\chi_1'\}\subseteq\Gamma$ or $\{\chi\c=\dforall s_0\chi_0,\chi\c=\dnot\chi_0'\}\subseteq\Gamma$, whence $\Gamma\Seq\Delta$ follows by the unique readability lemma. 
    Thus we may assume ($\forall_\T$L) was applied to obtain $\Gamma,\T\chi\Seq\Delta$ and so there are terms $s_1$, $\chi_1$ and $t$ such that $\{\chi\c=\dforall s_0\chi_0,\chi\c=\dforall s_1\chi_1\}\subseteq\Gamma$ and
    \[
    \Gamma,\T\chi,\T\chi_1[t/s_1]\Seq\Delta
    \]
    has a derivation with rank $(b',r)$ for some $b'<b$.
    Then there is some $r'<r$ for which the sequents
    \begin{align*}
        \Gamma\Seq s_0\c=s_1\wedge\chi_0\c=\chi_1 &&
        \Gamma\Seq\subdot d(\chi_0[v_i/s_0])\le\bar r'
    \end{align*}
    are truth-free derivable and so by term substitution we obtain a derivation of
    \begin{gather*}
        \Gamma,\T\chi,\T\chi_0[t/s_0]\Seq\Delta,
    \end{gather*} 
    with rank $(b',r)$. 
    Applying the induction hypothesis yields derivations of
    \begin{align*}
        \Gamma,\T\chi_0[t/s_0]\Seq\Delta &&
        \Gamma\Seq\Delta,\T\chi_0[ v_i/s_0]
    \end{align*}
    with ranks $((a+b')\cdot2,r)$ and $((a'+b)\cdot2,r)$ respectively. 
    Substituting $t$ for $v_i$ in the second derivation and applying (Cut$^{r'}_\T$) yields a derivation of $\Gamma\Seq\Delta$ with rank $((a+b)\cdot2,r)$.
\end{proof}
\begin{corollary}
    If the language of $\theory S$ does not contain $\T$ then 
    $\CT^*[\theory S]$ is a conservative extension of $\theory S$. 
\end{corollary}
\subsection{Obstacles}
It remains to embed $\CT$ into $\CT^*$. 
Consider, for example, a derivation of the form
\begin{prooftree}
    \AxiomC{$\vdots$}\noLine
    \UnaryInfC{$\Seq\T\phi$}
        \AxiomC{$\T\phi\Seq\T\phi$}
            \RightLabel{($\vee_\T$R)}
        \UnaryInfC{$\T\phi\Seq\T(\phi\dvee\phi)$}
                \RightLabel{(Cut$_\T$)}
    \BinaryInfC{$\Seq\T(\phi\dvee\phi)$}
\end{prooftree}
If the left-most sub-derivation is cut-free then the conclusion is also trivially derivable without cuts (simply apply the rule ($\vee_\T$R) to the conclusion of the left sub-derivation). 
Thus the cut in the above derivation could be assigned a rank of 1 regardless of the logical complexity of $\phi$. 
This can be explained by the fact that the complexity of any formula appearing under the truth predicate in the conclusion of the above cut (namely an instantiation of the term $\phi\vee\phi$ by closed terms) has complexity no greater than one plus the complexity of the cut formula (that is $\phi$).
It is also easy to see that this phenomenon holds for many deeper derivations. 
However, this manner of assigning cut rank is not sufficiently robust when it comes to derivations containing multiple cuts. 
We take the next derivation (the presentation of which has been intentionally simplified) as an example of the problem.
\begin{prooftree}
  \AxiomC{$\vdots$}\noLine
  \UnaryInfC{$\T\phi(\bar a),\T\phi(\bar b)\Seq\Gamma$}
    \LeftLabel{($\forall_\T$L)}
  \UnaryInfC{$\T\forall x\phi,\T\phi(\bar b)\Seq\Gamma$}
    \LeftLabel{($\forall_\T$L)}
  \UnaryInfC{$\T\forall x\phi\Seq\Gamma$}
      \AxiomC{$\vdots$}\noLine
      \UnaryInfC{$\Seq\Gamma,\T\phi(\dot x)$}
    \RightLabel{($\forall_\T$R)}
      \UnaryInfC{$\Seq\Gamma,\T\forall x \phi$}
        \RightLabel{(Cut$_\T$)}
  \BinaryInfC{$\Seq\Gamma$}
\end{prooftree}
The standard reduction lemma technique transforms the above derivation into the following in which cuts are on formulæ with intuitively lower complexity.
\begin{prooftree}
    \AxiomC{$\vdots$}\noLine
    \UnaryInfC{$\T\phi(\bar a),\T\phi(\bar b)\Seq\Gamma$}
        \AxiomC{$\vdots$}\noLine
        \UnaryInfC{$\Seq\Gamma,\T\phi(\bar a)$}
            \LeftLabel{(Cut$_\T$)}
    \BinaryInfC{$\T\phi(\bar b)\Seq\Gamma$}
        \AxiomC{$\vdots$}\noLine
        \UnaryInfC{$\Seq\Gamma,\T\phi(\bar b)$}
            \LeftLabel{(Cut$_\T$)}
    \BinaryInfC{$\Seq\Gamma$}
\end{prooftree}
The critical question is how to assign a rank to each of the two cuts in the second derivation that is strictly smaller than the rank given to the cut in the first derivation.
Assuming $a$ is different from $b$, the rank associated to the bottom cut must take into account the rank that is assigned to $\T\phi(\bar b)$ in the left sub-derivation as after an application of the cut reduction procedure to the top-most cut the intuitive complexity of the formula represented by $\phi(\bar b)$ may have increased. 
This is especially relevant if the sub-derivation contains other applications of the cut rule to ``sub-formulæ''  of $\phi(\bar b)$, $\phi(\bar a)$ or $\phi(\dot x)$.
Thus, if there is an appropriate way to assign ranks to occurrences of the truth predicate so the natural reduction procedure can be proven to succeed, it will require a deep analysis of the derivation as a whole.

The core idea is to provide a method to replace the term $\phi$ by a new term $\gn {B_\phi}$ that encodes a formula of  $\L^+$ with bounded logical complexity.
This formula will be chosen so that $\phi$ provably encodes a substitution instance of $B_\phi$.
In the case of the previous example, if the left-most sub-derivation is actually cut-free with height $n$ then $B_{\forall x\phi}$ can be chosen with complexity bounded by $\size\Gamma\cdot 2^n$, this being the longest possible chain of terms following the sub-formula relation induced by the derivation.
The complexity of $B_{\forall x\phi}$ will, in general, also be at least $ n$ so that each relevant occurrence of a sub-formula of $\phi$ in the derivation can be replaced by the corresponding sub-formula of $B_{\forall x\phi}$.
If the same choice suffices for the occurrence of $\forall x\phi$ in the right sub-derivation then this single occurrence of cut has been collapsed into a form available in $\CT^*$.
\section{Approximations}\label{sec:Approximations}

Recall the language $\L^+$ which extends $\L$ by countably many fresh predicate symbols
\[
    \mathcal P=\{\P^i_j\mid \text{$i$, $j<\omega$ and $\P^i_j$ is a predicate symbol of arity $i$}\}
\]
and a new propositional constant $\epsilon$. 
The additional predicate symbols enable us to explicitly reduce the complexity of formulæ that occur under the truth predicate in $\CT$-derivations.
This is achieved by the use of {\em approximations}, an idea that was utilised by Kotlarski et al in \cite{KKL81}.

An {\em assignment} is any function $g\colon X\to\L^+$ such that $X\subseteq\mathcal P$ is a finite set and for every $i$, $j$, if $\P^i_j\in X$ then $g(\P^i_j)$ is a formula with arity $i$. 
Given an assignment $g$ and an $\L^+$ formula $\phi$, we write $\phi[g]$ for the result of replacing each predicate $\P^i_j(s_1,\ldots,s_i)$ occurring in $\phi$ by $g(\P^i_j)(s_1,\ldots,s_i)$, if $g(\P^i_j)$ is defined, and $\epsilon$ otherwise.
If $\b\phi=(\phi_0,\ldots,\phi_m)$ and $\b\psi=(\psi_0,\ldots,\psi_m)$ are two sequences of closed $\L^+$ formulæ we say $\b\phi$ {\em approximates} $\b\psi$ if there exists an assignment $g$ such that $\psi_i=\phi_i[g]$ for each $i\le m$. 

For a given sequence $\b\phi$ of $\L^+$, a collection of approximations of $\b\phi$ are distinguished. The {\em $n$-th approximation} of $\b\phi$, defined below, is a particular approximation to $\b\phi$ that has logical complexity no more than $\lh{\b\phi}\cdot{2^n}$, where $\lh{\b\phi}$ denotes the number of elements in $\b\phi$. 
\subsection{Occurrences and parts}\label{sec:Occurrences}
Let $w$, $z$, $z_1$, $z_2$, \dots be fresh variable symbols.
Given a formula $\phi$ of $\L$ we first define a formula $\bar\phi$ of $\L\cup\{w\}$ in two steps: $\phi^*$ is the result of replacing in $\phi$ every free variable by $w$, and $\bar\phi$ is obtained from $\phi^*$ by replacing each term in which the only variable that occurs is $w$, by $w$.
Thus any term occurring in $\bar\phi$ is either simply the variable $w$ or contains a bound occurrence of a variable different from $w$.

For each formula $\phi$, we let $O(\phi)$ denote the set of {\em occurrences} of $\phi$, pairs $(\psi,s)$ such that $\psi$ is a formula of $\L\cup\{w,z\}$ in which the variable $z$ occurs exactly once, $s$ is a term of $\L\cup\{w\}$ which is free for $z$ in $\psi$ and $\phi=\psi[s/z]$. 
Notice that if $(\psi,s)\in O(\bar\phi)$ then $s=w$. 

The construction of $\bar\phi$ and $O(\phi)$ are such that for each formula $\phi$ of $\L$ there is a uniquely determined function $t_\phi\colon O(\bar\phi)\to \Term_{\L}$ for which $\phi$ is the result of replacing within $\bar\phi$ each occurrence of the variable $w$ by the appropriate value of $t_\phi$.
We call two formulæ $\phi$, $\psi$ {\em equivalent}, written $\phi\sim\psi$, if $\bar{\phi}=\bar{\psi}$.
\begin{lemma}\label{lem:Template}
    Let $\Phi$ be a set of $\L$ formulæ such that for every $\phi$, $\psi\in\Phi$, ${\phi}\sim{\psi}$. Then there is some number $l$ and formula $\vartheta_\Phi(z_1,\ldots,z_l)$, called the {\em template} of $\Phi$, such that for every $\phi\in\Phi$ there are terms $s_1$, \ldots, $s_l$ so that $\phi=\vartheta_\Phi(s_1,\ldots,s_l)$.
\end{lemma}
\begin{proof}
Suppose $\Phi$ is a set of formulæ satisfying the hypotheses of the lemma.
Notice that $O(\bar\phi)=O(\bar\psi)$ for every $\phi$, $\psi\in\Phi$, so $O(\Phi)$ has a natural definition as $O(\bar\phi)$ for some $\phi\in\Phi$.
The functions $\{t_\phi\mid\phi\in\Phi\}$ induce an equivalence relation $E_\Phi$ on $O(\Phi)$ by setting
\[
    (\chi,s)\mathbin{E_\Phi}(\psi,t) \iff \text{for every $\phi\in\Phi$, $t_\phi(\chi,s)=t_\phi(\psi,t)$.}
\]
Let $l$ be the number of $E_\Phi$-equivalence classes in $\Phi$.
For each $\phi\in\Phi$, the function $t_\phi$ is constant on $O(\Phi)/{E_\Phi}$, whence $\vartheta_\Phi(z_1,\dots,z_l)$ is easily defined.
\end{proof}

If $\b\phi=(\phi_0,\phi_1,\ldots,\phi_k)$ is a non-empty sequence of $\L$ formulæ, then the {\em set of parts of $\b\phi$}, $\Pi(\b\phi)$, is the collection of pairs $(\psi,\chi)$ such that $\psi$ is a formula of $\L\cup\{\epsilon\}$ in which $\epsilon$ occurs exactly once, $\chi$ is a formula of $\L$ and for some $i\le n$, $\phi_i$ is the result of replacing $\epsilon$ by $\chi$ in $\psi$. 
Notice that $\size{\Pi(\b\phi)}<k\cdot 2^{d({\b\phi})}$ where $d({\b\phi})$ denotes maximal logical complexity of formulæ occurring in $\b\phi$ with atomic formulæ having depth 0. 

We now define an ordering $\prec$ on $\Pi(\b\phi)$ as follows. $(\phi,\chi)\prec(\phi',\chi')$ just in case there exists $\psi\in\L\cup\{\epsilon\}$ such that $\phi'[\psi/\epsilon]=\phi$ and $\psi[\chi/\epsilon]=\chi'$. Informally, this means that $\phi[\chi/\epsilon]=\phi'[\chi'/\epsilon]$ and the occurrence of $\epsilon$ in $\phi$ corresponds to some sub-formula of $\chi'$. 
Note that this definition of $\prec$ is more refined than the ordering also denoted $\prec$ employed in \cite{KKL81}. The reasons for this will be highlighted later.
The {\em depth} of a pair $(\phi,\chi)\in\Pi(\b\phi)$, denoted $d(\phi,\chi)$, is its (reverse) order-type in $\prec$, that is the number of logical connectives and quantifiers between $\phi$ and the occurrence of $\epsilon$ in $\phi$.
Making use of $\prec$ and $\sim$ the following sets can be defined.
\begin{align*}
    \Pi^{0}(\b\phi,n)
        =\{(\phi,\chi)\in\Pi(\b\phi)&\mid d(\phi,\chi)\le n\}
        \\
    \Pi^{m+1}(\b\phi,n)
        =\{(\phi,\chi)\in\Pi(\b\phi)&\mid\exists (\phi_1,\chi_1)\in\Pi^{m}(\b\phi,n)\,\exists(\phi_0,\chi_0)\in\Pi^{0}(\b\phi,n)
        \\&\quad
        \wedge {\chi_0}\sim{\chi_1}\wedge (\phi,\chi)\prec (\phi_1,\chi_1)
        \\&\quad
        \wedge d(\phi,\chi)-d(\phi_1,\chi_1)\le n-d(\phi_0,\chi_0)\}
\end{align*}
The requirement ``$\exists(\phi_0,\chi_0)\in\Pi^{0}(\b\phi,n)$'' serves only to ensure the set $\Pi^{m+1}(\b\phi,n)$ does not grow too large. 
Thus $\Pi^{m+1}(\b\phi,n)$ consists of those parts of $\b\phi$ that are approximated by some $(\phi_1,\chi_1)$ in $\Pi^{m}(\b\phi,n)$ such that 
\begin{enumerate}\renewcommand\theenumi{\roman{enumi}}\renewcommand\labelenumi{\theenumi)}
    \item the template of $\chi_1$ occurs somewhere in $\b\phi$ with depth at most $n$, and 
    \item the depth of $(\phi,\chi)$ is regulated by the depth of $(\phi_1,\chi_1)$.
\end{enumerate}

\subsection{Approximating formulæ}\label{sec:ApproxFormulae}
The first crucial observation is that 
if $(\phi,\chi)\in\Pi^{m}(\b\phi,n)$ then there exists $(\phi',\chi')\in\Pi^{0}(\b\phi,n)$ with $\chi\sim{\chi'}$.
As a result, if
\[
    (\phi_0,\chi_0)\prec(\phi_1,\chi_1)\prec\cdots\prec(\phi_k,\chi_k)
\]
and $(\phi_i,\chi_i)\in\Pi^{m}(\b\phi,n)$ for every $i\le k$ then $k<\lh{\b\phi}\cdot2^n$, whence
\begin{gather}\label{eqn:PiBounded}
    (\phi,\chi)\in \Pi^m(\b\phi,n) \text{ implies } d(\phi,\chi) \le \lh{\b\phi}\cdot2^n
\end{gather}
and so $\size{\Pi^{m}(\b\phi,n)}\le 2^{\lh{\b\phi}\cdot2^n}$ for every $m$.
Since these bounds are independent of $m$, it follows there exists $k$ such that $\Pi^{k}(\b\phi,n)=\Pi^{k+1}(\b\phi,n)$. 

Based on the choice of $k$ two further sets are defined:
\begin{align*}
    \Gamma(\b\phi,n)=\{\psi\in\L&\mid\exists\phi(\phi,\psi)\in\Pi^{k}(\b\phi,n)\},
    \\
    \Gamma_I(\b\phi,n)=\{\psi\in\L&\mid\text{$\exists\phi$ $(\phi,\psi)$ is $\prec$-minimal in $\Pi^k(\b\phi,n)$}\}.
\end{align*}

Let $\Gamma_I^\sim(\b\phi,n)$ be the set of $\sim$-equivalence classes of $\Gamma_I(\b\phi,n)$ and suppose $\Phi\in\Gamma_I^\sim(\b\phi,n)$. 
We denote by $\vartheta_\Phi(z_1,\ldots,z_{l_\Phi})$ the {\em template} of $\Phi$ as determined in lemma~\ref{lem:Template}, and for each $\phi\in\Phi$ let $s^\phi_1$, \dots, $s^\phi_{l_\Phi}$ denote the terms for which $\phi=\vartheta_\Phi(s^\phi_1,\dots,s^\phi_{l_\Phi})$.

Utilising this notation a function $F_{\b\phi,n}\colon\Gamma(\b\phi,n)\to\L^+$ can be defined by recursion through $\prec$.
Fix some enumeration $\Phi_0$, \dots, $\Phi_n$ of the elements of $\Gamma_I^\sim(\b\phi,n)$, and let $a_j$ denote the number of arguments of the template $\vartheta_{\Phi_j}$.
If $\psi\in\Gamma_I(\b\phi,n)$ then either $\psi$ is atomic, whence we define $F_{\b\phi,n}(\psi)=\psi$, or $\psi\in\Phi_j\in\Gamma_I^\sim(\b\phi,n)$, whence $F_{\b\phi,n}(\psi)$ is chosen to be the formula $\P_j^{a_j}(s_1^\psi,\ldots,s_{a_j}^\psi)$. 
In the case $\psi\in\Gamma(\b\phi,n)\setminus\Gamma_I(\b\phi,n)$, $F_{\b\phi,n}(\psi)$ is defined to commute with the external connective or quantifier in $\psi$.

Now the {\em n-th approximation} of $\b\phi=(\phi_0,\ldots,\phi_k)$ is defined to be the sequence
\[
    F_{\b\phi,n}(\b\phi)=(F_{\b\phi,n}(\phi_0),\ldots, F_{\b\phi,n}(\phi_k)).
\]
These approximations have some nice features. For instance
\begin{lemma}\label{lem:1}
    Let $\b\phi$ be a sequence of $\L^+$-formulæ. Then the $i$-th approximation to $\b\phi$ is an approximation of $\b\phi$ and  an approximation of the $j$-th approximation whenever $i\le j$.
\end{lemma}
\begin{lemma}\label{lem:1.5}
    Every occurrence of a predicate symbol from $\mathcal P$ in the $n$-th approximation of $\b\phi$ has depth at least $n$ in $\b\phi$. Moreover, every member of the $n$-th approximation of $\b\phi$ has logical depth no greater than ${\lh{\b\phi}\cdot 2^n}$.
\end{lemma}
\begin{lemma}\label{lem:1.7}
    Suppose $(\b\phi',\b\psi')$ is an approximation of $(\b\phi,\b\psi)$ such every element $\chi\in\b\phi'\cup\b\psi'$ has logical complexity at most $n$. Then $(\b\phi',\b\psi')$ is an approximation of the $n$-th approximation of $(\b\phi,\b\psi)$.
\end{lemma}

The upper bound of lemma~\ref{lem:1.5} holds on account of \eqref{eqn:PiBounded}.
A consequence of the previous lemmas is the following.
\begin{lemma}\label{lem:KKLvee}
    If $(\b\phi',\psi'_0\vee\psi'_1)$ is the $n$-th approximation of $(\b\phi,\psi_0\vee\psi_1)$ and $m<n$ then the $m$-th approximation of $(\b\phi,\psi_i)$ is an approximation of $(\b\phi',\psi'_i)$.
\end{lemma}
Similarly we obtain:
\begin{lemma}\label{lem:KKLnot}
    If $(\b\phi',\lnot\psi')$ is the $n$-th approximation of $(\b\phi,\lnot\psi)$ and $m<n$ then the $m$-th approximation of $(\b\phi,\psi)$ is an approximation of $(\b\phi',\psi')$.
\end{lemma}
\begin{lemma}\label{lem:KKLforall}
    If $(\b\phi',\forall x\phi')$ is the $n$-th approximation of $(\b\phi,\forall x\phi)$ and $m<n$ then for every $a<\omega$ the $m$-th approximation of $(\b\phi,\phi[\bar a/x])$ is an approximation of $(\b\phi',\phi'[\bar a/x])$.
\end{lemma}
\subsection{Approximating sequents}\label{sec:ApproxSequents}
We begin by noting that all the definitions and results of the previous section can be formalised and proved within $\ID+\exp$.
Thus we fix the following formal notation.
\begin{enumerate}
  \item Gödel coding is expanded to sequences by letting $\gn{\b\phi}$ denote the term $(\gn{\phi_0},\ldots,\gn{\phi_m})$ if $\b\phi=(\phi_0,\ldots,\phi_m)$.
  \item  $(r)_i=s$ if $r$ encodes a sequence of length $k\ge i$ and $s$ is the $i$-th element of the sequence.
  \item If $\b s=(s_0,\ldots,s_m)$ and $\b t=(t_0,\ldots,t_n)$ are two sequences $\b s\concat \b t$ expresses the sequence $(s_0,\ldots, s_m,t_0,\ldots,t_n)$.
  In the case $m=n$ we introduce the following further abbreviations.
  \begin{enumerate}
    \item $\b s\c=\b t$ abbreviates $\bigwedge_{i\le m}(s_i\c=t_i)$;
    \item $\b s[g]$ abbreviates the sequence of terms $(s_0[g],\ldots,s_m[g])$;
    \item $\subdot F_{ r,u}(\b s)$ abbreviates the sequence of terms $(\subdot F_{ r,u}(s_0),\ldots,\subdot F_{ r,u}(s_m))$;
    \item $\subdot d(\b s)\le u$ abbreviates the formula $\bigwedge_{i\le m}\subdot d(s_i)\le u$.
  \end{enumerate}
  \item $s[g]\c=t$ expresses that either $g$ is not an assignment and $s\c=t$ or $g$ is an assignment and $t$ is the result of replacing within the $\L^+$ formula $s$, each occurrence of the predicate symbol $\P^i_j$ by $g(\P^i_j)$ if defined, otherwise by $\epsilon$.
  \item $\subdot F_{r,k}(s)=t$ expresses that there exists a sequence $\b\phi$ and $\psi\in\Gamma(\b\phi,k)$ such that $r=\gn{\b\phi}$, $s=\gn{\psi}$ and $t=\gn{F_{\b\phi,k}(\psi)}$; if there is no sequence of $\L_\T$-formulæ $\b\phi$ such that $r = \gn {\b\phi}$ then $s=t$.
\end{enumerate}
Note that the last point above expands to apply to complex equations involving multiple occurrences of sequences. 
So, for instance, $\subdot F_{r,u}(\b s)[g]=\subdot F_{r',u'}(\b t)$ is shorthand for the formula $\bigwedge_{i\le m}\subdot F_{r,u}(s_i)[g]=\subdot F_{r',u'}(t_i)$.

Collecting together the results of the previous section we have:
\begin{lemma}\label{lem:Formal1}
The following sequents are derivable in $\IDexp$.
\begin{align*}
    \emptyset&\Seq (x\dvee y)[z]\c=(x[z]\dvee y[z]),\\
    \emptyset&\Seq (\dnot x)[z]\c=\dnot (x[z]),\\
    \emptyset&\Seq (\dforall xy)[z]\c=\dforall x(y[z]),\\
    \emptyset&\Seq (y(x/w))[z]\c=(y[z])(x/w),\\
    (x)_i\c=y \dvee z 
        &\Seq \subdot F_{x,w+1}(y\dvee z)\c=
        \subdot F_{x,w+1}(y)\dvee \subdot F_{x,w+1}(z),\\
    (x)_i\c=\dnot y 
        &\Seq \subdot F_{x,w+1}(\dnot y)\c=
        \dnot \subdot F_{x,w+1}(y),\\
    (x)_i\c=\dforall y z 
        &\Seq \subdot F_{x,w+1}(\dforall y z)\c=
        \dforall y(\subdot F_{x,w+1}(z)),\\
    \emptyset&\Seq \subdot F_{x,w}(y_0\concat y_1\concat y_2)\c=
        \subdot F_{x,w}(y_0\concat y_2\concat y_1),\\
    \emptyset&\Seq \subdot d(\subdot F_{x,z}(\b s))
        \le \subdot{lh}(x)\cdot 2^z.
\end{align*}
\end{lemma}
\begin{lemma}\label{lem:Formal2}
There is a term $g$ with variables $w$, $x$, $y$ and $z$ such that the following sequents are truth-free derivable in $\IDexp$.
\begin{align*}
    \emptyset&\Seq\subdot d(g)\le \subdot{lh}(x)\cdot 2^z,\\
    y<z,w\c=x&\Seq \subdot F_{w,y}(u)[g]\c=\subdot F_{x,z}(u),\\
    y<z,x\c= x'\concat(x_0\dvee x_1) ,w\c=x'\concat x_i
        &\Seq\subdot F_{w,y}(w)[g]\c=
            \subdot F_{x,z}(w),\\
    y<z,x\c=x'\concat(\dnot x_0),w\c=x'\concat x_0
        &\Seq\subdot F_{w,y}(w)[g]  \c=\subdot F_{x,z}(w),\\
    y<z,x\c=x'\concat(\dforall x_0 x_1),w\c=x'\concat\subn&(x_1,x_2,u)\\
        &\Seq\subdot F_{w,y}(w)[g]\c=
            \subdot F_{x,z}(x')\concat \subn(\subdot F_{x,z}(x_2),x_1,u),\\
    w\c=x\concat w',\forall u(\subdot d(\subdot F_{w,y}(u))\le z)
        &\Seq\subdot F_{w,y}(x)[g]\c=\subdot F_{x,z}(x).
\end{align*}
\end{lemma}
The first sequent of lemma~\ref{lem:Formal2} formalises lemma~\ref{lem:1.5}, the second lemma~\ref{lem:1}, the third lemma~\ref{lem:KKLvee}, the penultimate line formalises lemma~\ref{lem:KKLforall}, expressing that the $y$-th approximation to $(\b\phi,\phi[a/x_2])$ can be viewed as an approximation of the $z$-th approximation to $(\b\phi,\forall x\phi)$ whenever $y<z$, and the final line combines lemmata~\ref{lem:1.7} and~\ref{lem:1.5}.

Thus tying in approximations with derivations we have:
\begin{lemma}\label{lem:Approx}
    Let $\Gamma$, $\Delta$ be sets consisting of arithmetical formulæ, and $\b\phi$, $\b\psi$ be sequences of terms.
    If $\Gamma,\T\b\phi\Seq\Delta,\T\b\psi$ is derivable then for every term $g$,
    \[
        \Gamma,\T \b\phi[g]\Seq\Delta,\T \b\psi[g]
    \]
    is derivable with the same truth bound.
    Moreover, if the first derivation contains no $\T$-cuts, neither does the second.
\end{lemma}
The lemma is not difficult to prove. However, we require a more general version that applies also to derivations featuring $\T$-cuts. The next lemma achieves this.
\begin{lemma}\label{lem:Approx2}
    Let $\Gamma$, $\Delta$, $\b\phi$ and $\b\psi$ be as in the statement of the previous lemma.
    If the sequents $\Gamma,\T\b\phi\Seq\Delta,\T\b\psi$ and $\Gamma\Seq \subdot d(g)<\bar k$ are derivable with truth ranks $(a,r)$ and $(0,0)$ respectively, the sequent 
\[
    \Gamma,\T\b\phi[g]\Seq\Delta,\T\b\psi[g]
\]
is derivable with truth rank $(a, r+k)$.
\end{lemma}
\begin{proof}
    The only non-trivial case is if the last rule is (Cut$^l_\T$) for some $l<r$. 
    So suppose $a=a'+1$ and we have the following derivation
    \begin{prooftree}
    \AxiomC{$\Gamma,\T\b\phi,\T\chi\Seq\Delta,\T\b\psi$}
    \AxiomC{$\Gamma,\T\b\phi\Seq\Delta,\T\chi,\T\b\psi$}
    \AxiomC{$\Gamma\Seq \subdot d(\chi)\le \bar l$}
    \RightLabel{(Cut$^l_\T$)}
    \TrinaryInfC{$\Gamma,\T\b\phi\Seq\Delta,\T\b\psi$}
\end{prooftree}  
with the two left-most premises derivable with truth rank $(a',r)$ and the right-most with rank $(0,0)$. 
    By the induction hypothesis, the sequents $\Gamma,\T\b\phi[g],\T\chi[g]\Seq\Delta,\T\b\psi[g]$ and $\Gamma,\T\b\phi[g]\Seq\Delta,\T\chi[g],\T\b\psi[g]$ are both derivable with rank $(a',r+k)$. 
    Since the sequent $\Gamma\Seq \subdot d(g)\le\bar k$ is derivable with rank $(0,0)$, so is 
    \[
        \Gamma\Seq \subdot d(\chi[g])\le \bar l+\bar k,
    \]
     whence the rule (Cut$^{l+k}_\T$) yields the desired sequent.
\end{proof}

\subsection{Approximating derivations}\label{sec:ApproxDeriv}

All that remains is to replace derivations in $\CT$ by approximations with bounded depth. 
Given a sequent $\Gamma,\T\b s\Seq\Delta,\T\b t$, its {\em $u$-th approximation} is the sequent 
    $\Gamma,\T(\subdot F_{\b s\concat \b t,\bar u}\b s)
        \Seq\Delta,\T(\subdot F_{\b s\concat \b t,\bar u}\b t)$.
Let $H$ be the function
\[
    H(k,n)={n\cdot2^{k}}.
\]
By lemma~\ref{lem:1.5} the $k$-th approximation of $\b\phi$ has depth at most $H(k,\lh{\b\phi})$.

The following lemmas hold for arbitrary derivations in $\CT^*[\theory S]$.

\begin{lemma}\label{lem:CT*forall}
    Suppose $a,r,m,n,k<\omega$, $\Gamma$ and $\Delta$ are finite sets of $\L$-formulæ, $\b\phi$ and $\b\psi$ are sequences of terms and $\psi$ is a term, none of which contain $x$ free and such that $lh(\b\phi)+lh(\b\psi)=n$. 
    If the $k$-th approximation to $\Gamma,\T\b\phi\Seq\Delta,\T\b\psi,\T(\psi(\dot x))$ is derivable with rank $(a,r)$ then there is a derivation with rank $(a+1,r+H(k+1,n+1))$ of the $(k+1)$-th approximation to $\Gamma,\T\b\phi\Seq\Delta,\T\b\psi,\T(\dforall \gn x\psi)$.
\end{lemma}
\begin{proof}
    Let $\b\chi=\b\phi\concat\b\psi\concat(\psi(\dot x))$. Then assumption of the lemma is that the sequent
    \[
        \Gamma,\T(\subdot F_{\b \chi,\bar k}\b\phi)
            \Seq\Delta,\T(\subdot F_{\b \chi,\bar k}\b\psi),
                \T(\subdot F_{\b\chi,\bar k}(\psi(\dot x)))
    \]
    is derivable with rank $(a,r)$.
    Let $g(x,y,z)$ be the term given by lemma~\ref{lem:Formal2} and $g'=g(\b\chi,\bar k,\bar k+1)$. 
    Lemma~\ref{lem:Approx2} implies there is a derivation with rank $(a,r+H(k+1,n+1))$ of the sequent
    \[
        \Gamma,\T(\subdot F_{\b \chi',\bar k+1}\b\phi)
            \Seq\Delta,\T(\subdot F_{\b \chi',\bar k+1}\b\psi),
                \T(\subdot F_{\b\chi,\bar k}(\psi(\dot x))[g'])
    \]
    where $\b\chi'=\b\phi\concat\b\psi\concat\forall x\psi$.
    Combining this derivation with those of lemmata~\ref{lem:Formal1} and~\ref{lem:Formal2} and using only arithmetical cuts, yields a derivation of the sequent
    \[
        \Gamma,\T(\subdot F_{\b \chi',\bar k+1}\b\phi)
            \Seq\Delta,\T(\subdot F_{\b \chi',\bar k+1}\b\psi),
                \T(\subdot F_{\b\chi',\bar k+1}(\psi)(\dot x))
    \]
    with rank $(a,r+H(k+1,n+1))$, whence ($\forall_\T$R) and lemma~\ref{lem:Formal1} yield that
    \[
        \Seq\Delta,\T(\subdot F_{\b \chi',\bar k+1}\b\psi),
                \T(\subdot F_{\b\chi',\bar k+1}(\forall \gn x\psi))
    \]
    is derivable with rank $(a+1,r+H(k+1,n+1))$.
\end{proof}

The same holds for the other scenarios:
\begin{lemma}\label{lem:CT*vee}
    If the $k$-th approximation to $\Gamma,\T\b\phi\Seq\Delta,\T\b\psi,\T\psi_i$ is derivable with rank $(a,r)$ then the $(k+1)$-th approximation of $\Gamma,\T\b\phi\Seq\Delta,\T\b\psi,\T(\psi_0\dvee\psi_1)$ is derivable with rank $(a+1,r+H(k+1,n))$, where $n=lh(\b\phi)+lh(\b\psi)+1$.
\end{lemma}
\begin{lemma}\label{lem:CT*Cut}
    Let $n=lh(\b\phi)+lh(\b\psi)$ and suppose $r\le H(k,n+1)$.
    If the $k$-th approximation to the sequents $\Gamma,\T\b\phi\Seq\Delta,\T\b\psi,\T\chi$ and $\Gamma,\T\b\phi,\T\chi\Seq\Delta,\T\b\psi$ are derivable with rank $(a,r)$ then the $H(k,n+1)$-th approximation of $\Gamma,\T\b\phi\Seq\Delta,\T\b\psi$ is derivable with rank 
    \[
        (a+1,H(k,n+1)+H(H(k,n+1),n))).
    \]
\end{lemma}
\begin{proof} 
    Let $N=H(k,n+1)$, $\b\omega=\b\phi\concat\b\psi$, $\b\omega'=\b\omega\concat\chi$. By lemma~\ref{lem:Formal1} there is a truth-free derivation of $\emptyset\Seq\subdot d(\subdot F_{\b\omega',\bar k}(x))\le \bar N$, so the sequent 
    \[
        \Gamma,\T(\subdot F_{\b\omega',\bar k}\b\phi)\Seq\Delta,\T(\subdot F_{\b\omega',\bar k}\b\psi)
    \]
    has a derivation with rank $(a+1,\max\{r,N\})$. 
    Let $g$ be given by lemma~\ref{lem:Formal2} and set $g'=g(\b\omega',\b\omega,\bar k,\bar N)$. Thus, lemma~\ref{lem:Formal2} entails 
    \[
        \emptyset\Seq 
            (\subdot F_{\b\omega',\bar k}\b\omega)[g']\c=
                \subdot F_{\b\omega,\bar N}\b\omega
    \]
    is truth-free derivable, whence we apply lemma~\ref{lem:Approx2} to obtain a derivation with rank $(a+1,\max\{r,N\}+H(N,n))$ of the sequent
    \[
        \Gamma,\T(\subdot F_{\b\omega,\bar N}\b\phi)\Seq\Delta,\T(\subdot F_{\b\omega,\bar N}\b\psi).
    \]
\end{proof}
\section{Proofs of the main theorems}\label{sec:main}
We now have all the ingredients for the {bounding lemma}, that permits the interpretation of derivations in $\CT[\theory S]$ as derivations in $\CT^*[\theory S]$. 
The next lemma generalises the statement of lemma~\ref{lem:BoundingIntro} by incorporating the relevant bounds.
\begin{lemma}[Bounding lemma]\label{lem:Bounding}
    There are recursive functions $G_1$ and $G_2$ such that for every $a,n<\omega$, if $lh(\b\phi)+lh(\b\psi)\le n$ and the sequent $\Gamma,\T\b\phi\Seq\Delta,\T\b\psi$ is derivable in $\CT[\theory S]$ with truth depth $a$, then its $G_1(a,n)$-th approximation is derivable in $\CT^*[\theory S]$ with rank $(a,G_2(a,n))$.
\end{lemma}
\begin{proof}
The idea is to copy the $\CT[\theory S]$ derivation into $\CT^*[\theory S]$ replacing the rule (Cut$_\T$) by (Cut$_\T^k$) for $k$ determined inductively. 
The functions $G_1$ and $G_2$ are defined according to the bounds obtained in the previous section:
\begin{align*}
    G_1(0,n)
        &=0,    \\
    G_1(m+1,n)
        &=H(G_1(m,n+1),n+1),\\
    G_2(m,n)
        &=G_1(m+1,m+n).
\end{align*}
We argue by induction on $a$. Suppose the last rule applied to obtain $\Gamma,\T\b\phi\Seq\Delta,\T\b\psi$ is a non-arithmetical cut on $\T\chi$ and that this derivation has height $a+1$. Let $\b\omega=\b\phi\concat\b\psi$ and $\b\omega'=\b\omega\concat\chi$.
The induction hypothesis implies 
that the $G_1(a,n+1)$-th approximations to
\begin{align*}
    \Gamma,\T\b\phi,\T\chi\Seq\Delta,\T\b\psi &&
    \Gamma,\T\b\phi\Seq\Delta,\T\chi,\T\b\psi
\end{align*}
are each derivable in $\CT^*[\theory S]$ with rank $(a,G_2(a,n+1))$.
By lemma~\ref{lem:CT*Cut} there is a derivation with height $a+1$ of the $G_1(a+1,n)$-th approximation to $\Gamma,\T\b\phi\Seq\Delta,\T\b\psi$. This derivation has cut rank bounded by $G_2(a+1,n)$
so we are done. The other cases are similar and follow from applications of lemmas~\ref{lem:CT*forall} and \ref{lem:CT*vee}.
\end{proof}

A combination of lemmas~\ref{lem:Approx} and \ref{lem:Bounding} implies that $\CT[\theory S]$ permits the elimination of all $\T$-cuts.
\begin{corollary}\label{cor:cutelimination}
    If $\Gamma \Seq \Delta$ is derivable in $\CT[\theory S]$ then it is derivable without $\T$-cuts.
\end{corollary}

\subsection{Proof of theorem~\ref{thm:1}}
Let $\phi$ be an arithmetical theorem of $\CT[\theory S]$. 
By the Embedding Lemma, the sequent $\emptyset\Seq\phi$ has a derivation within $\CT[\theory S]$.
Lemma~\ref{lem:Bounding} implies that the same sequent is derivable in $\CT^*[\theory S]$ and the cut elimination theorem for $\CT^*[\theory S]$ shows $\emptyset\Seq\phi$ is derivable without truth cuts.
But this derivation is also a derivation within $\theory S$. 
Notice that this final derivation has height bounded by $2^a_{2\cdot G_1(a+1,a+1)}$, where $a$ bounds the height of the original derivation of $\emptyset\Seq\phi$ in $\CT[\theory S]$, $G_1$ is as defined in the proof of the Bounding Lemma, and $2^n_m$ represents the function of hyper-exponentiation: $2^n_0=2^n$ and $2^n_{m+1}=2_m^{2^n}$.
Thus this reduction can be formalised within $\IDexp_1$.

\subsection{Proof of theorem~\ref{thm:2}}
Let $\theory S$ and $\pred D$ be as given in the statement of the theorem and let $U$ be the finite set of $\L\cup\{p\}$ formulæ associated with the $\theory S$-schema $\pred D$.
We will show that the Bounding lemma naturally extends to  provide a reduction of the theory $\CT[\theory S]+\forall x(\pred Dx\rightarrow \T x)$ into the extension of $\CT^*[\theory S]$ by the rule 
\begin{prooftree}
    \AxiomC{$\Gamma\Seq\Delta,\pred Ds$}
    \RightLabel{($\pred D$)}
    \UnaryInfC{$\Gamma\Seq\Delta,\pred \T s$}
\end{prooftree}
Despite the fact that all cuts in this latter theory remain bounded, unlike $\CT^*[\theory S]$ the theory will not in general support the cut elimination procedure.
Nevertheless, conservativity over $\theory S$ can be achieved by considering the additional assumptions.

Suppose $d$ is a derivation with truth depth $a$ of the truth-free sequent $\Gamma\Seq\Delta$ in the expansion of $\CT[\theory S]$ by the rule ($\pred D$).
By redefining the functions $G_1$ and $G_2$ so that $G_1(0,n)$ bounds the logical depth of the (finitely many) formulæ in $U$ for each $n$, the proof of the Bounding Lemma can be carried through to obtain a derivation with rank $(a,G_2(a,0))$ of the same sequent in the system expanding $\CT^*[\theory S]$
by a variant of ($\pred D$):
\begin{prooftree}
    \AxiomC{$\Pi,\T\b\phi\Seq\Sigma,\T\b\psi,\pred D\sigma$}
    \LeftLabel{($\pred D_{\b\omega}$)}
    \UnaryInfC{$\Pi,\T\b\phi\Seq\Sigma,\T\b\psi,\pred \T(\subdot F_{\b\omega,\bar k}\sigma)$}
\end{prooftree}
where $\Pi$ and $\Sigma$ are truth-free, $k=G_1(a,0)$ and $\b\omega=\b\phi\concat\b\psi\concat\sigma$.

Let $d^*$ denote this derivation. Fix $n$ such that for each instance of ($\pred D_{\b\omega}$) occurring in $d^*$, $\lh{\b\omega}<n$, and set $U^+$ to be the finite set of instantiations of formulæ from $U$ by $\L$-formulæ that have logical depth at most $G_2(a,n)$. 
It follows that the sequent $\pred Dx,\d d(x)<\overline{G_2(a,n)}\Seq \{x=\gn \phi\mid \phi\in U^+\}$ is derivable in $\theory S$. 
Because the sequent
    $\sigma\Seq \T\gn\sigma$    
is derivable in $\CT[\theory S]$ for each $\L$-sentence $\sigma$ we may deduce 
\[
  \pred Dx\Seq\T (\subdot F_{\gn {\b\omega},\bar k}x)  
\]
is derivable in $\CT[\theory S]$ whenever $\lh{\b\omega}<n$. 
Thus $d^*$ can be interpreted in $\CT[\theory S]$ and an application of theorem~\ref{thm:1} completes the proof.
\section{Conservativity, interpretability and speed-up}\label{sec:otherresults}

The following instance of theorem~\ref{thm:2} is particularly revealing:
\begin{corollary}\label{cor:Conserve}
    Let $\pred{Ind}_\L$ be the formula expressing that $x$ is the code of the universal closure of an instance of $\L$-induction.
    Then $\CT[\PA]+\forall x(\pred{Ind}_\L x\rightarrow\T x)$ conservatively extends $\theory{PA}$.
\end{corollary}
Corollary~\ref{cor:Conserve} effectively shows the limit of what principles can be conservatively added to $\CT[\PA]$.
It is well known that extending $\CT[\theory {PA}]$ by induction for formulæ involving the truth predicate (even only for bounded formulæ) allows the deduction of the global reflection principle $\forall x(\Bew_\theory {PA} x\rightarrow \T x)$, and hence the schema of reflection $\Bew_\PA\gn\phi\rightarrow \phi$, a statement not provable in $\theory{PA}$.

An analogous result holds also for other first-order systems such as set theories. 
For example, the above corollary still holds if $\PA$ is replaced by Zermelo-Fraenkel set theory and $\pred{Ind}$ is replaced by a formula recognising instances of the separation and replacement axioms. 
Expanding the axiom schemata of $\CT[\theory{ZF}]$ to apply also to formulæ involving the truth predicate, however, yields a non-conservative extension.\footnote{Assuming $\theory{ZF}$ is consistent.}

We conclude the paper with two corollaries that are specific to a proof-theoretic treatment of $\CT[\theory S]$.
\begin{corollary}\label{cor:Speedup}
    $\CT[\theory S]$ attains at best hyper-exponential speed-up over $\theory S$.
\end{corollary}
To restate Corollary~\ref{cor:Speedup}, every $\L$-theorem of $\CT[\theory S]$ is derivable in $\theory S$ with at most  hyper-exponential increase in the length of the derivation.
The upper-bound results from the fact the conservativeness of $\CT[\theory S]$ over $\theory S$ can be established within $\IDexp_1$.

Fischer, in~\cite{Fis09}, discusses a further consequence of a formalised conservativeness proof for $\CT$.
\begin{lemma}[Fischer~\cite{Fis09}]\label{lem:Fischer}
    If $\PA \vdash \forall x (\Sent_\L x \wedge \Bew_{\CT[\theory S_0]}x \rightarrow \Bew_{\theory S_0}x)$ for every $\theory{I\Sigma}_1 \subseteq \theory S_0\subseteq \PA$ then $\CT[\PA]$ is relatively interpretable in $\PA$.\footnote{We refer the reader to, e.g., \cite{Fis09} for a definition of {\em relatively interpretable}.}
\end{lemma}
Combining this with theorem~\ref{thm:1} therefore yields
\begin{corollary}\label{cor:Fischer}
    If $\theory S\subseteq \PA$ then $\CT[\theory S]$ is relatively interpretable in $\PA$.
\end{corollary}
\section*{Acknowledgements}
I would like to thank Albert Visser, Ali Enayat, Kentaro Fujimoto and Volker Halbach for their helpful comments on earlier versions of this paper.
This work was supported by the Arts and Humanities Research Council UK grant no.\ AH/H039791/1.

\end{document}